\documentclass[a4paper,12pt, leqno]{amsart}
\usepackage{amsmath,amssymb,amsthm,amsfonts,amstext,amscd}
\usepackage[margin=2.5cm]{geometry}
\usepackage[english]{babel}
\usepackage[ansinew]{inputenc}
\usepackage[T1]{fontenc}
\usepackage{lmodern}
\usepackage{newlfont}
\usepackage{longtable}
\usepackage[all]{xy}
\usepackage[colorlinks=true, citecolor=red, linkcolor=blue]{hyperref}
\usepackage{color}
\usepackage{lineno}

\newcommand{\nc}{\newcommand}

\theoremstyle{plain}
\newtheorem{theorem}{\sc Theorem}[section]
\newtheorem{thm}[theorem]{\sc Theorem}
\newtheorem{lem}[theorem]{\sc Lemma}

\newtheorem{prop}[theorem]{\sc Proposition}
\newtheorem{cor}[theorem]{\sc Corollary}

\newtheorem{rem}[theorem]{\sc Remark}

\theoremstyle{definition}
\newtheorem{define}[theorem]{\sc Definition}

\renewcommand{\refname}{\center{REFERENCES}}

\newcommand{\vfi}{\varphi}

\newcommand{\Z}{\mathbb{Z}}

\nc{\GG}{\mathfrak{G}}

\nc{\LL}{\mathcal{L}}
\nc{\JJ}{\mathcal{J}}

\nc{\fN}{\mathfrak{N}}
\nc{\fH}{\mathfrak{H}}
\nc{\fF}{\mathfrak{F}}
\nc{\cH}{\mathcal{H}}
\nc{\cG}{\mathcal{G}}
\nc{\cN}{\mathcal{N}}
\nc{\cK}{\mathcal{K}}
\nc{\cT}{\mathcal{T}}

\nc{\pgs}{\mathfrak{p}\mathfrak{g}\mathfrak{s}}
\nc{\pcr}{\mathfrak{p}\mathfrak{c}\mathfrak{r}}
\nc{\wg}{\widehat{g}}
\nc{\wh}{\widehat{h}}
\nc{\wx}{\widehat{x}}
\nc{\wy}{\widehat{y}}
\nc{\wk}{\widehat{k}}

\nc{\Char}{\operatorname{char}}
\nc{\Ker}{\operatorname{Ker}}
\nc{\Imm}{\operatorname{Im}}
\nc{\Aut}{\operatorname{Aut}}
\nc{\Cl}{\operatorname{Cl}}
\nc{\Orb}{\operatorname{Orb}}
\nc{\noreq}{\trianglelefteq}
\nc{\lcm}{\operatorname{lcm}}

\advance\textheight by \topskip
\title[The non-abelian Tensor Square of Finite Metacyclic Groups]
{The non-abelian Tensor Square of Finite Metacyclic Groups}
\author[Canella]{Juliana Silva Canella}
\address{Institute of Exact and Natural Sciences, 
Universidade Federal do Par\'a, Par\'a-PA, 66075-110 Brazil}
\email{jscanella@ufpa.br}
\author[Rocco]{Nora\'i Romeu Rocco}
\address{Departamento de Matem\'atica-IE, Universidade de Bras\'ilia,
Bras\'ilia-DF, 70910-900 Brazil}
\email{norai@unb.br}
\date{\today}

\begin{document}

\begin{abstract}
In this paper, we investigate the group $\nu(G)$, an extension of the non-abelian tensor square $G$ by the direct product $G\times G$, in order to determine a presentation of  $G \otimes G$ when $G$ is a general finite metacyclic group,   $G=g(a,b; m,n, r, s)$, with $m$ odd. A presentation of $\nu(G)$ is obtained from that of $G$ and, consequently, we describe the appropriate relevant  sections of $\nu(G)$, such as the non-abelian tensor square, the exterior square $G \wedge G$ and the Schur multiplier, $M(G)$.
\end{abstract}
\maketitle
\section{Introduction} 

The {\it non-abelian tensor square} $G \otimes G$ of a group $G$ as introduced by Brown and Loday \cite{BL1, BL} is defined to be the group generated by all symbols $g \otimes h, \, g, h \in 
 G$, subject
to the relations
$$g{g_1} \otimes h = (g^{g_1} \otimes h^{g_1} ) (g_1 \otimes h) \quad \text{and} \quad  g \otimes h {h_1} = (g \otimes  h_1)(g^{h_1} \otimes h^{h_1})$$
for all $g, g_1, h, h_1 \in G.$ Here and in the sequel, for elements $x, y$ in a group $G$ we write $x^y$ to denote the conjugate $y^{-1} x y$; the commutator of $x$ and $y$ is then the element $[x, y] := x^{-1} x^y$ and our commutators are left normed, $[x, y ,z] := [[x,y], z]$ for any $x, y, z \in G.$

The group $G \otimes G$ is a particular case of a non-abelian tensor product $G \otimes H$, of groups $G$ and $H$, under the assumption that $G$ and $H$ acts one on each other and on itself by conjugation. 

Given its relationship to several other relevant homology functors, the computation of this object for a given group $G$ is of great interest.

The study of the non-abelian tensor square from a group theoretical point of view started with Brown, Johnson and Robertson in \cite{BJR}. In that paper the authors compute $G \otimes G$ for all non-abelian groups of order up to $30$ using Tietze transformations in order to simplify the presentations. This method, however, is not appropriate for larger groups, since it involves $|G|^2$ generators and $2|G|^3$ relations. 

A method that has been used is based on the concept  of a \textit{biderivation} (or a \textit{crossed pairing}): let $G$ and $L$ be groups; 
a {\it biderivation} from $G \times G$ to $L$ is a function $f: G \times G \to L$ such that
\begin{center}
    (i) $f(aa_1,b) = f(a^{a_1},b^{a_1})f(a_1,b);$   \\[2mm]
    (ii) $f(a,bb_1) = f(a,b_1)f(a^{b_1},b^{b_1})$
\end{center}
for all $a,a_1,b,b_1 \in G$. Any biderivation $f: G \times G \to L$ determines a unique  homomorphism $$f^*: G \otimes G \to L$$ such that the diagram \begin{displaymath}
  \xymatrix{
    {G \times G} \ar[rr]^{j} \ar[drr]_{f}
    && {G \otimes G} \ar[d]^{f^*}\\
    && {L} 
  }
\end{displaymath}
is commutative, where $j: G \times G \to G \otimes G, \; (g, h) \mapsto g \otimes h$ is the canonical biderivation arising from the definition of $G \otimes G.$ Therefore, to determine the non-abelian tensor square via biderivations, as done by many authors \cite{Bacon, BK, BKM, Kappe1, BJR, Johnson, Kappe, MM, V}, means to conjecture both the group $L$ and the biderivation $f$, which in general is not an easy task.

On the other hand, since the defining relations of $G \otimes G$ are abstractions of commutator relations in groups, in \cite{Rocco1} the second named author suggested another way to obtain the non-abelian tensor square by making use of the following construction.
Let $G$ be a group and let $G^{\varphi}$ be an isomorphic copy of $G$ via $\varphi : g \mapsto g^{\varphi}, \, \forall g \in G$. Then we define the group $\nu(G)$ to be
\[
\nu(G):= \langle G \cup G^{\varphi} \mid [g_1,g_2^{\varphi}]^{g_3}=[g_1^{g_3},(g_2^{g_3})^{\varphi}] =[g_1,g_2^{\varphi}]^{g_3^{\varphi}}, g_i \in G\rangle.
\]
In \cite{EL}, Ellis and Leonard considered a similar construction. According to  \cite[Proposition 2.6]{Rocco1}, the subgroup $\Upsilon(G)=[G,G^{\varphi}]$ of $\nu(G)$ is isomorphic to the non-abelian tensor square $G \otimes G$, which explains the initial motivation to study the group $\nu(G)$. Actually, the group $\nu(G)$ admits the decomposition  $\nu(G) = (\Upsilon(G) \cdot G) \cdot G^{\varphi}$, where the dots mean (internal) semidirect products. Moreover, thanks to the work of several mathematicians, this construction has become powerful machinery for computing and handling the non-abelian tensor square $G \otimes G$, specially when the group $G$ is polycyclic (see \cite{BFM, BlyMor, EN, EL, Rocco2}).  

The \textit{diagonal subgroup} $\Delta(G) := \langle [g, g^{\varphi}] \mid g \in G\rangle \leq [G,G^\vfi]$ is central in  $\nu(G)$ (see \cite[Remark 2]{Rocco2}). The section $[G, G^{\varphi}]/ \Delta(G)$  is isomorphic with the non-abelian exterior square  $G \wedge G$, studied by Miller in \cite{Miller} (\cite[p. 1984]{Rocco2}).

The evaluation of $G \otimes G$ for finite split metacyclic $G$ was carried out via biderivations by Brown, Johnson and Robertson in \cite{BJR}, and by Johnson in \cite{Johnson}; the same technique was used by Beuerle and Kappe in \cite{Kappe1} to compute $G \otimes G$ for infinite metacyclic groups $G$.   

The application $\rho: \nu(G) \to G$, $g\mapsto g$, $g^\varphi \mapsto g$, for every $g \in G$, induces an epimorphism of $\nu(G)$ over $G$ whose kernel, $\theta (G)$, centralizes
$[G, G^\varphi]$ in $\nu(G)$.  The restriction of $\rho$ to $[G, G^\varphi]$ is the \textit {derived map}, $\rho': [G, G^\varphi] \to G'$, $[g, h^\varphi] \mapsto [g, h]$ for all $g, h\in G$, whose kernel, $\mu(G) \; (= \Upsilon(G) \cap \theta(G))$, is then central in $\nu(G)$ and is quite related to Homology Theory \cite {BlyMor}; for instance, $\mu (G)/ \Delta (G)$ is isomorphic with the {\it Schur multiplier} of $G$, $M(G)$.

In this paper, we explore the commutator approach via the group $\nu(G)$ to determine both $\nu(G)$ and $G \otimes G$ for a general finite metacyclic $G$. As a consequence we  also describe the exterior square, $G\wedge G$ and the Schur multiplier, $M(G)$, as sections of $\nu(G)$. 

The paper is organized as follows. In Section 2 we establish some notation and properties involving a finite metacyclic group $G=g(a,b;m,n,r,s)$ (see Definition~\ref{def:metacyclic_grp}). In Section 3 we briefly describe the group $\nu(G)$ and establish some structural results. In Section 4 we give a presentation for group $\nu(G)$ and obtain other of its sections such as $G\otimes G$ and $M(G)$ for a finite metacyclic group $G$ with $m$ odd.

\section{Preliminaries} \label{sec:2}

In this section we establish some notation and address basic results concerning finite metacyclic groups. 

Let $m$ be a non-negative integer. The ring of integers modulo $m$ is denoted by $\mathbb{Z}_{m}$. For non-negative integers $\alpha$ and $\beta$, not both zero, let $(\alpha, \beta):= \gcd(\alpha, \beta)$ and $(0,0):= 0$; the $\lcm(\alpha, \beta)$ will be denoted by $[\alpha, \beta]$, whenever this does not cause confusion with commutators in groups. For $m > 1$, we write $\mathcal{U}_{m}$ for the set $\{r \in \mathbb{Z} \, \mid \, 1 \leq r \leq m-1 \; \text{and} \; \gcd(r,m) = 1\};$  if $m=0$, then $\mathcal{U}_{0}=\{-1, 1\}$. If $m > 1$ and $r \in \mathcal{U}_{m}$, let $t$ be the unique  integer in $\mathcal{U}_{m}$ such that $r t \equiv 1 \pmod m$. %

A {\it metacyclic group} is an extension of a cyclic group by a cyclic group. From now on we shall consider the following general finite metacyclic group fixed in the next definition. 
\begin{define} \label{def:metacyclic_grp}
Let $m$ and $n$ be positive integers. We define the group 
$G = g(a,b; m,n,r,s)$ to be the finite metacyclic group given by the power-conjugate presentation
\[
G:= \langle a, b \, \mid \, a^m=1, \; b^n=a^s, \; a^b=a^r \rangle, \quad \text{where} 
\] 
\begin{equation} \label{eq:metacyclic1} 
r^n \equiv 1\pmod m \quad \text{and} \quad s(r-1) \equiv 0 \pmod m.
\end{equation}
\end{define}
\noindent The above conditions on the 4-tuple of integers $(m,n,r,s)$ determines a consistent polycyclic presentation to define the  metacyclic group $G$ of order $mn$, a result due to H\"{o}lder [\cite{Za}, Theorem 20]. We also see that every element $g \in G$ can be written in the normal form $g = b^\beta a^\alpha$, where $0 \leq \alpha \leq m-1, \; 0 \leq \beta \leq n-1$; in particular we always chose the parameter $s$ such that $0 \leq s < m.$ From now on we fix this notation, according to Definition~\ref{def:metacyclic_grp}.

\begin{rem} \label{rem:one}    
For the purposes of this article we assume that our groups $G = g(a,b;m,n,r,s)$ are non-abelian (in particular, non-cyclic); thus we have $(m,s) \neq 1$ and $1 < r \in \mathcal{U}_m$. In addition, if $r = m-1$, then $m > 2$ and $n$ is even, regardless of the parity of $m$; also, if $m$ is odd, then $G = g(a, b; m, n, m-1, 0)$ is a split extension, because $m \mid s (m-2)$; this fact also follows immediately from the Schur-Zassenhaus Theorem \textnormal{[\cite{Rob}, \S $9.12$]}\label{SchurZas}. 
\end{rem}
As usual we write $H'$ for the derived subgroup $[H, H]$ and $Z(H)$ for the center of a group $H$. For an arbitrary element $x \in H$ we write $o(x)$ for the order of $x$ and $o'(x)$ for the order of the coset $xH'$ in the abelianized group $H^{\text{ab}} = H/H'$. 
\begin{lem} \label{lem:orders1}
Let $G = g(a,b;m,n,r,s)$ be a finite metacyclic group of order $mn$, $t= m/(m, r-1)$ and let $l = |r \pmod m|$ be the multiplicative order of $r \pmod m$. Then
\begin{enumerate}
\item[$(i)$] $o(a)  = m \quad \text{and} \quad  o(b) = n \dfrac{m}{(m,s)};$ \\
\item[$(ii)$] $G' = \langle a^{r-1} \rangle \cong C_t \quad \text{and} \quad Z(G) = \langle a^t, b^l \rangle;$ \\
\item[$(iii)$] $o'(a) = (m, r-1) \quad \text{and} \quad  o'(b) = n \dfrac{(m, [s,r-1])}{(m,s)}.$
\end{enumerate}
\end{lem}
\begin{proof} Part $(i)$ follows directly from the fact that $A := \langle a \rangle$ has order $m$, $|G|= |\langle b \rangle A| = mn$ and $|\langle b \rangle \cap A| = o(a^s)$.  Observe that $G' = \langle [a, b] \rangle = \langle a^{r-1} \rangle$ from part $(ii)$. Now, from the defining relations we have that $a^t$ is the least power of $a$ that centralizes $b$ (and also $a$), while $b^l$ is the least power of $b$ that commutes with $a$, $a^{b^l} = a^{r^l} = a$  (this result can also be found in Sim \cite[Lemma 2.7]{Sim}). As for part $(iii)$ we observe that  $o'(b) = [\langle b \rangle : \langle b \rangle \cap G']$ is equal to the index $[\langle b \rangle : \langle a^q \rangle]$, where $q=[s,r-1]$, since  $\langle b \rangle \cap G' = \langle a^s \rangle \cap \langle a^{r-1} \rangle = \langle a^q \rangle$. This gives the desired expression for $o'(b)$.
\end{proof}

The following number theoretical function, borrowed from Beuerle and Kappe \cite{Kappe1},  is useful in estimating the orders of the elements of a generating set of the non-abelian tensor square.
\begin{define} \textnormal{\cite[Definition 2.1]{Kappe1}} \label{def:func_Em} 
Let $m, r$ be integers with $m>0$ and $r \in \mathcal{U}_m$. Define $E_{m} : \mathcal{U}_{m} \times \mathbb{Z} \to \mathbb{Z}$ to be the function given by
\begin{equation} 
E_{m}(r, x) :=
\left \{
\begin{array}{cc}
x, & \ if \,\, r = 1 \,\, or \,\, x = 0 \nonumber\\
1+r+\cdots + r^{x-1}, & if \,\, r \neq 1 \,\, \text{and} \,\, x>0 \nonumber\\
-r^x E_{m}(r, -x), & if \,\, r \neq 1 \,\, \text{and} \,\, x< 0 \nonumber
\end{array}
\right.
\end{equation}
\end{define}

In the next Lemma we establish some useful additional relations that hold for $G=g(a,b;m,n,r,s)$, $r>1$, which follow directly from the defining relations of $G$. The proofs will be omitted.   

\begin{lem}\label{rel.grupo.meta.finito}
Let $G = g(a,b;m,n,r,s)$ as in Definition~\ref{def:metacyclic_grp}. Then every element $g\in G$ can be written as $g=b^\beta a^\alpha$, where $0 \leq \alpha \leq m-1$ and $0 \leq \beta \leq n-1$. Furthermore, if $h = b^\delta a^\gamma \in G$, the following relations are satisfied in $G$:
\begin{enumerate} 
\item[$(i)$]   $(a^\alpha)^{b^\beta}=a^{\alpha r^\beta}$ and $(b^\beta)^{a^\alpha}=b^\beta a^{\alpha(1-r^{\beta})}$;
    \item[$(ii)$]
    $gh =b^{\beta + \delta} a^{\gamma + \alpha r^\delta}$;
    \item[$(iii)$] $h^g = b^{\delta}a^{\gamma r^{\beta} + \alpha (1-r^\delta)}$;
    \item[$(iv)$] 
    $g^\sigma = b^{\sigma\beta}a^{\alpha E_{m}(r^\beta, \, \sigma)}$ for all $\sigma \in \mathbb{Z}$. 
\end{enumerate}
	\end{lem}
    
Below we define another natural number, $k$, which is also  useful for estimating upper bounds for the orders of the generators of $G \otimes G$.
\begin{define}{\label{def:k}} For the 4-tuple $(m,n,r,s)$ defining a finite metacyclic group $G$, let
\begin{equation*} 
    k:=\left( \frac{m}{(m,s)}, 2\frac{[s,r-1]}{(m,s)}, n\frac{[s,r-1]^2}{(m,s)^2} , E_m(r, o(b)) \right). 
\end{equation*} 
\end{define} 

\begin{rem} Notice that if $m$ is odd then $k$ is also odd, regardless of the parity of $r \in \mathcal{U}_m$. 
\end{rem}

In this paper we shall consider only the groups $G=g(a,b; m,n,r,s)$ with $m$ odd. The cases in which $m$ is even are under preparation and will be addressed in a subsequent paper.

\vskip 10pt
\noindent
\section{The group $\nu(G)$ and some of its relevant sections}

Let $G^\vfi$ be an isomorphic copy of $G$ through the isomorphism $\vfi : g \mapsto g^\vfi$ for all $g \in G$ and consider the group 
\[
\nu(G):= \langle G \cup G^{\varphi} \mid [g_1,g_2^{\varphi}]^{g_3}=[g_1^{g_3},(g_2^{g_3})^{\varphi}] =[g_1,g_2^{\varphi}]^{g_3^{\varphi}}, g_i \in G\rangle.
\]
As observed in the introduction, the normal subgroup $\Upsilon(G) := [G, G^\vfi] \leq \nu(G)$ is canonically isomorphic with the non-abelian tensor square $G \otimes G$  through the isomorphism induced by $[g, h^\vfi] \mapsto g \otimes h$ (see \cite[Proposition 2.6]{Rocco1}). In view of this isomorphism we will write $[g, h^\vfi]$ instead of $g \otimes h$ and use this commutator approach for the purposes of this article. 

The group $\nu(G)$ admits the following decomposition,  where the dots mean (internal) semidirect products.  
\begin{equation} \label{eq:decomp_nu}
\nu(G) = (\Upsilon(G) \cdot G) \cdot G^{\varphi}. 
\end{equation}
Recall that the \textit{diagonal subgroup} $\Delta(G) := \langle [g, g^{\varphi}] \mid g \in G\rangle$ is central in  $\nu(G)$ (see \cite[Remark 3]{Rocco2}) and the section $[G, G^{\varphi}]/ \Delta(G)$  is isomorphic with the non-abelian exterior square  $G \wedge G$, studied by Miller in \cite{Miller} (see \cite[p. 1984]{Rocco2}).

There is an epimorphism $\rho : \nu(G) \to G, \; g \mapsto g, \; h^\vfi \mapsto h$ for all $g \in G, h^\vfi \in G^\vfi.$ We write $\Theta(G)$ for $\ker \rho$ and we have 
\[ \Theta(G) = \langle [g, \vfi] := g^{-1}g^\vfi \, \mid \, g \in G \rangle \quad \text{and} \quad \nu(G) = \Theta(G) \cdot G,
\]
In addition, it follows from the defining relations that $\Theta(G)$ centralizes $\Upsilon(G)$. Restriction of $\rho$ to $\Upsilon(G)$ gives the \textit{derived map} $\rho' : [g, h^\vfi] \mapsto [g, h],$ whose kernel $\ker \rho' = \Upsilon(G) \cap \Theta(G)$ is a central subgroup of $\nu(G)$, denoted by $\mu(G)$ (see \cite[Proposition 2.7]{Rocco2}). Thus we have:
 \begin{equation} \label{eq:mu}
   \dfrac{\Upsilon(G)}{\mu(G)} \cong G' \quad \text{and} \quad  \dfrac{\mu(G)}{\Delta(G)} \cong M(G), 
 \end{equation}
 where $M(G) \cong H_2(G, \Z)$ is the Schur multiplier of $G$. 
 It is worth noting that our map $\rho'$ is the derived map $\kappa$ in \cite{BJR}, where the  kernel $\ker \kappa$  is written $J_2(G).$
 
 Since $\mu(G) \leq Z(\Upsilon(G))$ it follows that $\Upsilon(G)$ will be abelian whenever $G'$ is cyclic. This is true, in particular, for all metacyclic groups and so we have
 \begin{cor} \label{cor:g'cyclic}
     The non-abelian tensor square of the group $G=g(a,b;m,n,r,s)$ is abelian, for any choice of the 4-tuple $(m,n,r,s).$
 \end{cor}

The following basic relations holding in $\nu(G)$ are consequences of the defining relations and commutator calculus (see, for instance,  \cite{BlyMor}, \cite{BJR},\cite{Rocco2}, \cite{Rocco1}). 
\begin{lem}
\label{lem:ppNu} For any group $G$ the following relations hold in $\nu(G)$:
\begin{itemize}
    \item[$(i)$] $[[g, h^\varphi ], [x, y^\varphi]] = [[g, h] , [x, y]^\varphi ]$, for all $g, h, x, y\in G$;
    \item[$(ii)$] $[g_1 , [g_2 , g_3 ]^\varphi] = [g_1 , g_2 , g_3^\varphi ]^{-1}$, for all $g_1, g_2, g_3\in G$;
    \item[$(iii)$]  $[g_1,g_2^\varphi]^{[g_3,g_4^\varphi]}=[g_1,g_2^\varphi]^{[g_3,g_4]}$ , for all $g_1, g_2, g_3, g_4\in G$;
    \item[$(iv)$] $[g_1,g_2^\varphi,g_3]=[g_1,g_2,g_3^\varphi]=[g_1,g_2^\varphi,g_3^\varphi]=[g_1^\varphi,g_2,g_3]=[g_1^\varphi,g_2,g_3^\varphi]=[g_1^\varphi,g_2^\varphi,g_3]$, for all $g_1,g_2,g_3\in G$;
    \item[$(v)$] $[g,g^\varphi]$ is central in $\nu(G)$, for all $g\in G$;
    \item[$(vi)$] $[g_1,g_2^ \varphi][g_2,g_1^\varphi]$ is central in $\nu(G)$, for all $g_1, g_2\in G$;
    \item[$(vii)$] $[g,g^\varphi] =1$, for all $g\in G'$.
\end{itemize}
\end{lem}

\begin{lem}\textnormal{[\cite{Rocco2}, Lemma $3.1$]} \label{lem:o'_pp} Let $G$ be a group and $g$, $h$ be arbitrary elements of $G$. Then:
\begin{itemize}
    \item[$(i)$] $[g,h^\varphi][h,g^\varphi]=[gh,(gh)^\varphi].[h,h^\varphi]^{- 1}.[g,g^\varphi]^{-1}\ \left( \in \Delta(G) \right)$;
\item[$(ii)$] $[g,h^\varphi][h,g^\varphi]=[h,g^\varphi][g,h^\varphi]$;
\item[$(iii)$] If $h\in G'$ (or if $g \in G'$) then $[g,h^\varphi][h,g^\varphi]=1$ and $[h,h^{\varphi}]=1$;
\item[$(iv)$] If $gG'=hG'$ then $[ g,g^\varphi]=[h,h^\varphi]$;
\item[$(v)$]  If $o'(g)$ or $o'(h)$ is finite then $o([ g,h^\varphi][h,g^\varphi])$ divides $(o'(g),o'(h))$;
\item[$(vi)$]  If $o'(h) $ is finite then $o([h,h^\varphi])$ divides $(o'(h)^2,2o'(h))$;
\item[$(vii)$] In particular, if $o'(h)$ is odd, then $o([h,h^\varphi])$ divides $o'(h)$. 
\end{itemize}
\end{lem}
The next result, found in \cite{Rocco2}, gives a simplified presentation for $\nu(H)$ when $H$ is a finite solvable group given by a polycyclic presentation $H = \langle \mathcal{S} \mid \mathcal{R} \rangle$ (see also \cite{BlyMor} for an extension of this result to polycyclic groups in general). 
\begin{theorem}\textnormal{[\cite{Rocco2}, Theorem $2.1$, Corollary $2.2$]} \label{Rocco.nu} Let $H$ and $H^\varphi$ be finite solvable groups with respective polycyclic generating sequences, $\mathcal{S} = \{a_1,a_2, \ldots, a_n\}$ and $\mathcal{S}^\vfi = \{b_1, b_2,\ldots, b_n \}$, with power-conjugate relations $\mathcal{R}$ and $\mathcal{R}^\vfi$, where $\varphi: H \to H^\varphi$ is an isomorphism such that $a_i \mapsto b_i (=a_i^{\vfi})$, $1\leq i\leq n$. Then
\begin{align*}
  (i) \;\; \nu(H) & \cong \langle a_i, b_i, 1 \leq i \leq n \, \mid \,  \mathcal{R}, \mathcal{R}^\vfi, [a_i,b_j]^{a_k}=[a_i^{a_k},b_j^{b_k}]=[a_i,b_j]^{b_k}, 1 \leq i,j,k\leq n \rangle; \\  
 (ii) \; \Upsilon(H) & =  \langle [a_i, b_j] \mid 1 \leq i, j \leq n \rangle.
\end{align*}
\end{theorem}
In particular, we can apply the above result to our group $G = g(a,b; m,n,r,s) = \langle \mathcal{S} \mid \mathcal{R} \rangle$, where here $\mathcal{S} =\{a, b\}$ and $\mathcal{R} = \{a^m, b^na^{-s}, a^ba^{-r} \},$ according to Definition~\ref{def:metacyclic_grp}, to obtain a simplified presentation of $\nu(G).$ To this end we need to control the actions by conjugation of the generators of $\nu(G)$ on the generators $[a, a^\vfi], [a, b^\vfi], [b, a^\vfi]$ and $[b, b^\vfi]$ of $\Upsilon(G) = [G, G^\vfi]$. Thus, let $\mathcal{W}$ be the following set of identities:
\begin{equation} \label{eq:conj_action}
\begin{split}
\mathcal{W} & = \{ [a, a^\vfi]^a = [a, a^\vfi]^{a^\vfi} = [a, a^\vfi], \;  
[a, a^\vfi]^b = [a, a^\vfi]^{b^\vfi} = [a^b, (a^b)^\vfi], \\
\  & \qquad [a, b^\vfi]^a = [a, b^\vfi]^{a^\vfi} = [a, (b^a)^\vfi], \; [a, b^\vfi]^b = [a, b^\vfi]^{b^\vfi} = [a^b, b^\vfi], \\ 
\ & \qquad [b,a^\vfi]^a = 
[b, a^\vfi]^{a^\vfi} = [b^a ,a^\vfi], \; 
[b,a^\vfi]^b = 
[b, a^\vfi]^{b^\vfi} = [b ,(a^b)^\vfi], \\
\ & \qquad [b,b^\vfi]^a = 
[b, b^\vfi]^{a^\vfi} = [b^a ,(b^a)^\vfi], \; 
[b,b^\vfi]^b = 
[b, b^\vfi]^{b^\vfi} = [b , b^\vfi] \}. 
\end{split}
\end{equation}

\begin{cor} \label{cor:poly:nu}
Let $G = g(a,b; m,n,r,s) = \langle \mathcal{S} \mid \mathcal{R} \rangle$ be the metacyclic group as above.
Then 
\[\nu(G) = \langle a, b, a^\vfi, b^\vfi \mid \mathcal{R}, \mathcal{R}^\vfi, \mathcal{W} \rangle.
\]
\end{cor}
\begin{rem} \label{rem:delta}
    Alternatively, we could  consider the generators $[a, a^\vfi], [b, b^\vfi], [a, b^\vfi] $ and $[a, b^\vfi][b, a^\vfi]$ of $\Upsilon(G)$. We observe that $\Delta(G) = \langle [a, a^\vfi], [b, b^\vfi], [a, b^\vfi][b, a^\vfi] \rangle$ (see \cite[Proposition 3.3]{Rocco2}), so that $\Upsilon(G) = \Delta(G) \langle [a, b^\vfi] \rangle $; this also shows that $\Upsilon(G)$ is abelian, by Corollary~\ref{cor:g'cyclic}, once $\Delta(G) \leq Z(\nu(G)).$ 
\end{rem}

\section{ The non-abelian tensor square of a finite metacyclic group and other sections of $\nu(G)$ }

In this section we first establish upper bounds (as accurate as possible) for the orders of the generators of $\Upsilon(G)$. To this end it is important to further expand the identities of the set $\mathcal{W}$. The identities in the following Lemma, whose proofs are by induction arguments using the relations $\mathcal{R}$ of the group $G=g(a,b; m,n,r,s)$ and Lemmas ~\ref{rel.grupo.meta.finito} and ~\ref{lem:ppNu}, actually are  combinations of results found in  \cite{BlyMor}, \cite{BJR}, \cite{Rocco2}, \cite{Rocco1}.  

\begin{lem}\label{lema.poten.nu.meta}\label{pp.nu}
Let $G=g(a,b; m,n,r,s)$. Then the following identities hold in $\nu(G)$ for all $\alpha, \beta \in \mathbb{Z}_+$:
	\begin{itemize}
\item[$(i)$] $[a,b^\varphi]^{a^\alpha}=[a,b^\varphi][a,a^\varphi]^{\alpha(r-1)}$,
\item[$(ii)$] $[a^\alpha,b^\varphi]=[a,b^\varphi]^\alpha[a,a^\varphi]^{\binom{\alpha}{2}(r-1)}$,
\item[$(iii)$]  $[a,b^\varphi]^{b^\beta}=[a,b^\varphi]^{r^\beta}[a,a^\varphi]^{\binom{r^\beta}{2}(r-1)}$,
\item[$(iv)$] $[a,(b^\beta)^\varphi]=[a,b^\varphi]^{E_m(r,\beta)}[a,a^\varphi]^{(r-1)\sum_{i=1}^{\beta-1}\binom{r^i}{2}}$,
\item[$(v)$] $[a^\alpha,b^\varphi]^{b^\beta}=[a,b^\varphi]^{\alpha r^\beta}[a,a^\varphi]^{(r-1)\binom{\alpha r^\beta}{2}}$,
\item[$(vi)$]  $[a^\alpha,(b^\beta)^\varphi]=[a,b^\varphi]^{\alpha E_{m}(r,\beta)}[a,a^\varphi]^{(r-1)\sum_{i=0}^{\beta-1}\binom{\alpha r^i}{2}}$,
\item[$(vii)$]  $[b,a^\varphi]^{a^\alpha}=[b,a^\varphi][a,a^\varphi]^{\alpha(1-r)}$,
\item[$(viii)$]  $[b,(a^\varphi)^\alpha]=[b,a^\varphi]^\alpha[a,a^\varphi]^{\binom{\alpha}{2}(1-r)}$,
\item[$(ix)$]  $[b,a^\varphi]^{b^\beta}=[b,a^\varphi]^{r^\beta}[a,a^\varphi]^{\binom{r^\beta}{2}(1-r)}$,
\item[$(x)$] $[b^\beta,a^\varphi]=[b,a^\varphi]^{E_m(r,\beta)}[a,a^\varphi]^{(1-r)\sum_{i=1}^{\beta-1}\binom{r^i}{2}}$,
\item[$(xi)$] $[b,(a^\alpha)^\varphi]^{b^\beta}=[b,a^\varphi]^{\alpha r^\beta}[a,a^\varphi]^{(1-r)\binom{\alpha r^\beta}{2}}$,
\item[$(xii)$]  $[b^\beta,(a^\alpha)^\varphi]=[b,a^\varphi]^{\alpha E_{m}(r,\beta)}[a,a^\varphi]^{(1-r)\sum_{i=0}^{\beta-1}\binom{\alpha r^i}{2}}$.
	\end{itemize}
 \end{lem}

The next Lemma relates the element $[b,b^\varphi]$ with other generators of $\Upsilon(G)$, independently of the $4$-uple $(m,n,r,s)$. As usual, we write $2 \parallel s$ to say that $2$ is the highest power of $2$ that divides $s$.

\begin{lem}\label{lema.meta.fin.n.s} Let $G=g(a,b;m,n,r,s)$. Then occurs in $\nu(G)$: 
$$\begin{cases} [b,b^\varphi]^n=[a,b^\varphi]^s=[b,a^\varphi]^s, & \mbox{if\ } 2\nmid s \ \mbox{or\ if\ } 4\mid s\\ [b,b^\varphi]^n=[a,b^\varphi]^s[a,a^\varphi]^{(r -1)}=[b,a^\varphi]^s[a,a^\varphi]^{(1-r)}, & \mbox{if\ } 2\parallel s.
\end{cases}
$$
\end{lem}
\begin{proof} By {Lemma \ref{lema.poten.nu.meta}}, items $(ii)$ and $(viii)$, \begin{eqnarray} &[b^n,b^\varphi ]=[a^s,b^\varphi]=[a,b^\varphi]^s[a,a^\varphi]^{\binom{s}{2}(r-1)}&\label{Lemma1}\\ &[b,(b^n)^\varphi]=[b,(a^s)^\varphi]=[b,a^\varphi]^s[a,a^\varphi] ^{\binom{s}{2}(1-r)}.&\nonumber \end{eqnarray} If $s$ is odd, $[a,a^\varphi]^{\binom{s}{2}(r-1)}=[a,a^\varphi]^{\frac{(s-1)}{2}s( r-1)}=1$, since $m \mid s(r-1)$. If $4 \mid s$, then $\binom{s}{2}$ is even, $[a,a^\varphi]^{\binom{s}{2}(r-1)}= 1$ and, Lemma~\ref{lem:o'_pp} (vi), $o([a,a^\varphi])\mid 2(r - 1)$.

If $s=2s'$ with $s'$ odd, we have odd number
$$\binom{s}{2}=\frac{2s'(2s'- 1)}{2}=s'(2s'-1).$$ Thus, by $(\ref{Lemma1})$, $$[a,a^\varphi]^{\binom{s}{2}(r-1)}=[a,a^\varphi ]^{s'(2s'-1)(r-1)}=[a,a^\varphi]^{[s'(2s'-1)-1](r-1)}[a,a ^\varphi]^{(r-1)}=[a,a^\varphi]^{(r-1)},$$ since $s'(2s'-1)-1$ is even, by Lemma~\ref{lem:o'_pp} (vi). The other identity follows in an analogous manner.\end{proof}

Henceforth we focus on $G=g(a,b;m,n,r,s)$ with $m$ \textbf{odd}. The following Proposition  establishes upper bounds for the orders of the generators of $\Upsilon(G)$ where $k$ is settled in Definition \ref{def:k}.

\begin{prop}\label{prop.cota.quadd.mrta.fin} Let $G=g(a,b;m,n,r,s)$  be a finite metacyclic group with  $m$ odd and $s>0.$ The following upper bounds hold for the orders of the generators of $[G, G^\varphi]$:
\begin{itemize}
    \item[$i)$] $o([a,a^\varphi])$  divides $(m,r-1);$
    \item[$ii)$] $o([b,b^\varphi])$  divides $n(k,r-1);$
    \item[$iii)$] $o([a,b^\varphi])$  divides $(m, E_m(r,o(b)), sk);$
    \item[$iv)$] $o([a,b^\varphi][b,a^\varphi])$ divides $\left(o'(a), o'(b), E_{m}(r,o'(b)), sk \right).$
\end{itemize}
\end{prop}
\begin{proof} $i)$ follows from Lemma~\ref{lem:o'_pp} $(vii)$, since $o'(a)=(m,r-1)$ is odd. 
\noindent As for item $ii)$, it follows from the fact that $o'(b)$ is finite and $$1=[b^{o(b)},b^{\varphi}]=[b,(b^{o(b)})^{\varphi}]=[b,b^{\varphi}]^{o(b)}.$$ Therefore, $[b,b^{\varphi}]$ has order dividing \begin{equation}\left(\frac{mn}{(m,s)}, {\frac{m^2n^2}{(m,s)^2}}, 2\frac{mn}{(m,s)}, 2\frac{n[s,r-1]}{(m,s)}, {\frac{n^2[s,r-1]^2}{(m,s)^2}}
	\right),\nonumber
	\end{equation} i.e., \begin{equation}
	n\left(\frac{m}{(m,s)}, 2\frac{[s,r-1]}{(m,s)}, n\frac{[s,r-1]^2}{(m,s)^2}\label{eq.o'}
	\right).\end{equation} 
By {Lemma \ref{lema.meta.fin.n.s}}, we should consider the numbers $n(r-1)$ and $nE_m(r,o(b))$ to bound the order of $[b,b^\varphi]$, since \begin{equation}
1=[a,b^\varphi]^{s(r-1)}=[a,b^\varphi]^{sE_m(r,o(b))}=[b,b^\varphi]^{n(r-1)}=[b,b^\varphi]^{nE_m(r,o(b))}\nonumber,
\end{equation} regardless of the parity of $s$. Thus, by (\ref{eq.o'}) together with the definition of $k$, we get the bound $n(k, r-1)$ for the $o([b,b^\varphi])$.

$iii)-iv)$ Using the relations established in {Lemma \ref{lema.meta.fin.n.s}}, we can prove itens $iii)$ and $iv)$ simultaneously.

From the identity \begin{eqnarray}
		1&=& [a^m, b^\varphi]=[a,b^\varphi]^m[a,a^\varphi]^{\binom{m}{2}(r-1)},\label{7}\end{eqnarray} we have that $[a,b^\varphi]^m=1$, since $[a,a^\varphi]^{r-1}=1$, by item $i)$.

Now, applying item $i)$ in the identity \begin{equation}
1=[a,(b^{o(b)})^\varphi]=[a,b^\varphi]^{E_m(r,o(b))}[a,a^\varphi]^{(r-1)\sum_{i=1}^{o(b)-1}\binom{r^i}{2}},\nonumber
\end{equation} we will obtain that $o([a,b^\varphi])\mid E_m(r,o(b))$. Analogously, we obtain the following relations: \begin{equation}
[a,b^\varphi]^{E_{m}(r,n)}=[b,a^\varphi]^{E_{m}(r,n)}=[a,a^\varphi]^s\label{id.s}\end{equation} and so $[a,b^\varphi]$ has order dividing $\left( m, E_m(r,o(b)) \right)$.

Applying item $i)$ again to the identities in {Lemma \ref{lema.meta.fin.n.s}}, we obtain \begin{equation}[b,b^\varphi]^n=[a,b^\varphi]^s=[b,a^\varphi]^s,\nonumber\end{equation} regardless of the choice of $s$. Since $[b,b^\varphi]^{nk}=1$, by item $ii)$, we have that \begin{equation} 1=[b,b^\varphi]^{nk}=[a,b^\varphi]^{sk}=[b,a^\varphi]^{sk},\nonumber\end{equation} and $k$ is odd, it also follows that \begin{equation}
([a,b^\varphi][b,a^\varphi])^{sk}=[a,b^\varphi]^{sk}[b,a^\varphi]^{sk}=1.\nonumber
\end{equation}

Finally, by {Lemma \ref{lema.poten.nu.meta}}, we have: \begin{eqnarray} 1&= &[a,(b^{o'(b)})^\varphi][b^{o'(b)},a^\varphi]\nonumber \\ &=&[a,b^\varphi] ^{E_{m}(r,o'(b))}[a,a^\varphi]^{(r-1)\sum_{i=0}^{o'(b)-1}\binom {r^i}{2}}\nonumber\\
 &&[b,a^\varphi]^{E_{m}(r,o'(b))}[a,a^\varphi]^{(1-r)\sum_{i=0}^{o '(b)-1}\binom{r^i}{2}}\nonumber\\ &=&\left( [a,b^\varphi][b,a^\varphi] \right)^{E_ {m}(r,o'(b))} \end{eqnarray}

 We do not need to analyze when $1=[a,(b^{o(b)})^\varphi][b^{o(b)},a^\varphi]$ because since $o'(b)\mid o(b)$, we have $E_{m}(r,o'(b))\mid E_{m}(r,o(b))$. In fact, consider $r^{o'(b)}=y$ and $o(b)=\theta o'(b)$. Then \begin{equation}
E_{m}(r^{o'(b)},\theta)=\frac{y^\theta-1}{y-1}=\frac{r^{\theta o'(b)}-1}{r^{o'(b)}-1}=\frac{r^{o(b)}-1}{r^{o'(b)}-1}= \dfrac{\dfrac{r^{o(b)}-1}{r-1}}{\dfrac{r^{o'(b)}-1}{r-1}}=\frac{E_{m}(r,o(b))}{E_{m}(r,o'(b))},\label{id.Em}\end{equation} since $E_{m}(r^{o'(b)},{\theta})\in \mathbb{Z}^\star_+$ or $n\mid o'(b)$. \end{proof}

\begin{rem}\label{rmk: diag}
From the proof of {Proposition \ref{prop.cota.quadd.mrta.fin}} above and by {Lemma \ref{lema.meta.fin.n.s}}, we obtain the following identities:
$$
\begin{cases}
[a,b^\varphi]^s=[b,a^\varphi]^s=[b,b^\varphi]^n\in \Delta(G),\\

[a,b^\varphi]^{E_m(r,n)}=[b,a^\varphi]^{E_m(r,n)}=[a,a^\varphi]^s\in \Delta(G).
\end{cases}
$$
\end{rem}

Having obtained the above upper  bounds on the orders of the generators of $\Upsilon(G)$, we are in a position to give a presentation for $\nu(G)$.

\begin{thm}\label{TEO.A}\label{thm: 1} Let $G=g(a,b; m,n,r,s)$ be a finite  metacyclic group with $m$ odd and define a group $M$ by the following presentation: \begin{eqnarray}\langle x_1, y_1, x_2, y_2, u, v, w, z \hspace{-.2cm}&\mid& \hspace{-.2cm} x_1^m=1, x_2^m=1, y_1^n=x_1^s, y_2^n=x_2^s, [x_1,y_1]=x_1^{r-1}, [x_2,y_2]=x_2^{r-1},\nonumber \\ 
&& \hspace{-.2cm} [x_1,y_2]=u, 
[x_1,x_2]=v, [y_1,y_2]=w, [y_1, x_2]=u^{-1}z, v^{o'(x_1)}=1,\nonumber \\ 
&& \hspace{-.2cm}  w^{n(k,r-1)}=1, u^{s}=w^n=(u^{-1}z)^s, v^s=(u^{-1}z)^{E_m(r,n)}=u^{E_m(r,n)},  \nonumber \\ 
&&\hspace{-.2cm}u^{(m,E_{m}(r,o(y_1)), sk)}=1, z^{(o'(x_1),o'(y_1),sk,E_m(r,o'(y_1)))}=1,\nonumber \\ 
&&\hspace{-.2cm} u^{x_1}=u^{x_2}=u, u^{y_1}=u^{y_2}=u^r, \ (v, w, z\ \mbox{central})\nonumber\rangle.
\end{eqnarray} Then, $M\cong \nu (G)$.
\end{thm}
\begin{proof} We will first prove that the group $M$ has $\nu(G)$ as a homomorphic image. 
Indeed, by Corollary \ref{cor:poly:nu}, we know that $\nu(G)$ is generated the set $\{a, b, a^{\varphi}, b^{\varphi} \}$, while $\Upsilon(G) = [G, G^{\varphi}]$, as an abelian subgroup of $\nu(G)$, is generated by $[a,a^\varphi]$, $[b,b^\varphi]$, $[a,b^\varphi][b,a^\varphi]$ and $[a,b^\vfi]$;   we also know that the first three of these commutators are central in $\nu(G)$.  Consider the map $\psi: M \to \nu(G)$ defined by 
$x_1 \mapsto a$, $y_1 \mapsto b$, $x_2 \mapsto a^\vfi$ and
   $y_2 \mapsto b^\vfi$. 
Thus, $u=[x_1, y_2]\mapsto [a,b^\vfi]$, $v=[x_1, x_2]\mapsto [a,a^\vfi]$, $w=[y_1, y_2]\mapsto [b,b^\vfi]$ and $z=[x_1, y_2][y_1, x_2]\mapsto [a,b^\vfi][b,a^\vfi]$ and, according to Proposition \ref{prop.cota.quadd.mrta.fin}, their relative orders are preserved by their respective images. In addition, the other relations in the given presentation  of $M$ are clearly preserved by $\psi$, in view of the subset $\mathcal{W}$ (\ref{eq:conj_action}) of the defining relations of $\nu(G)$, according to Corollary \ref{cor:poly:nu} and Lemma \ref{pp.nu}. Since $[a,a^\vfi]$, $[b,b^\vfi]$ and $[a,b^\vfi][b,a^\vfi]$ are central in $\nu(G)$, we conclude that all defining relations of $M$ are preserved by $\text{Im} \; \psi$ and thus, $\psi$ extends to an epimorphism.

Conversely, consider the map $\tau: \nu(G)\rightarrow M$ defined by $a \mapsto x_{1}$, $b \mapsto y_{1}$, $a^\varphi\mapsto x_{2}$ and $b^\varphi \mapsto y_{2}$. The presentation of $\nu(G)$ in Corollary \ref{cor:poly:nu} applies here and it is an easy exercise to check that those relations are satisfied by the presentation of $M$. It is readily seen that 
$\tau \psi = Id_M$ and $\psi \tau = Id_{\nu(G)},$ given that $M \stackrel{\psi}{\simeq} \nu(G).$ 
 \end{proof}

In view of the isomorphism given by Theorem~\ref{thm: 1} and its proof, we see that the subgroup $\Upsilon(G) \leq \nu(G)$ is isomorphic with the subgroup of  $M$ generated by $\{u, v, w, z\}$. As a consequence,  we obtain the following presentation for the non-abelian tensor square $G \otimes G \; (\simeq \Upsilon(G))$. 
\begin{cor} \label{cor: nats}
Let $A$ be the abelian group given by the presentation
 \begin{eqnarray} \langle u, v, w, z \hspace{-.2cm}&|& \hspace{-.2cm} u^{(m,E_{m}(r,o(b)), sk)}=1,v^{o'(a)}=1, z^{(o'(a),o'(b),sk,E_m(r,o' (b)))}=1,\nonumber \\ &&\hspace{-.2cm} w^{n(k,r-1)}=1, u^{s}=w^n=(u^{-1}z)^s, u^{E_m(r, n)}=v^s=(u^{-1}z)^{E_m(r,n)},\nonumber \\ &&\hspace{-.2cm} [u,v]=[u,w]=[u,z]=[v,w]=[v,z]=[w,z]=1\nonumber \rangle.  
 \end{eqnarray}
 Then $G \otimes G \simeq A.$
\end{cor}

It is worth noting that the presentation of $\nu(G)$ obtained in Theorem~\ref{thm: 1} is motivated, according to Corollary~\ref{cor: nats}, by the decomposition (\ref{eq:decomp_nu})  of the group $\nu(G)$ as semidirect products, $\nu(G) = ((\Upsilon(G) \cdot G)\cdot G^\vfi \simeq ((G \otimes G) \rtimes G) \rtimes G$, where the actions are those as in (\ref{eq:conj_action}), taking into account the bounds established in Proposition~\ref{prop.cota.quadd.mrta.fin} and Remark~\ref{rmk: diag}. Indeed, we can first construct an  extension starting with an abelian group $A$ (as given in Corollary~\ref{cor: nats}) by $G$, and then extending this semidirect product $A \rtimes G$ by the other copy of $G$.  

\begin{thm}\label{thr:2} Let $G=g(a,b; m,n,r,s)$ be a finite  metacyclic group with $m$ odd. Then the presentation of $\nu(G)$ given by Theorem~\ref{thm: 1} can be effectively constructed.  
\end{thm}
	\begin{proof} Let $A$ be the abelian group given by 
\begin{eqnarray} A=\langle u, v, w, z \hspace{-.2cm}&|& \hspace{-.2cm} u^{(m,E_{m}(r,o(b)), sk)}=1,v^{o'(a)}=1, z^{(o'(a),o'(b),sk,E_m(r,o' (b)))}=1,\nonumber \\ &&\hspace{-.2cm} w^{n(k,r-1)}=1, u^{s}=w^n=(u^{-1}z)^s, u^{E_m(r, n)}=v^s=(u^{-1}z)^{E_m(r,n)},\nonumber \\ &&\hspace{-.2cm} [u,v]=[u,w]=[u,z]=[v,w]=[v,z]=[w,z]=1\nonumber \rangle. \end{eqnarray} 
The actions of the generators of $G$ on those generators of $\Upsilon(G)$, according to (\ref{eq:conj_action}), suggest an application $f: G\rightarrow \Aut(A)$ defined on the generators by the following:

\begin{minipage}{.45cm}
\begin{eqnarray}
a&\longmapsto&a^f:A\longrightarrow A \nonumber\\
& &\hspace{.8cm} u\longmapsto (u)a^f:=u\nonumber\\
& &\hspace{.8cm} v\longmapsto (v)a^f:=v\nonumber\\
& &\hspace{.8cm} w\longmapsto (w)a^f:=w\nonumber\\
& &\hspace{.8cm} z\longmapsto (z)a^f:=z\nonumber
\end{eqnarray}	
\end{minipage}\hfil
\begin{minipage}{.45cm}
\begin{eqnarray}
\hspace{.5cm}b&\longmapsto&b^f:A\longrightarrow A \nonumber\\
& &\hspace{.8cm} u\longmapsto (u)b^f:=u^r\nonumber\\
& &\hspace{.8cm} v\longmapsto (v)b^f:=v\nonumber\\
& &\hspace{.8cm} w\longmapsto (w)b^f:=w\nonumber\\
& &\hspace{.8cm} z\longmapsto (z)b^f:=z\nonumber
\end{eqnarray}
\end{minipage}
\medskip
\medskip

\noindent {\bf Assertion 1}: $f$ extends to a homomorphism from $G$ to $\Aut(A)$.

{\bf Step 1}: $a^f$ and $b^f$ extend to homomorphisms of $A$.

The application $a^f$ fixes the elements $u$, $v$, $w$ and $z$ of $A$ and the application $b^f$ fixes the elements $v$, $w$ and $z$. Therefore, it is necessary to analyze the relations of $A$ to the elements $u^r$ and $u^{-r}z$:
\begin{equation}(u^r)^m=u^{mr}=1, \hspace{.7cm} (u^{-r}z)^m=u^{-mr}z^m=1.\nonumber\end{equation} Furthermore,\begin{equation}
(u^r)^s=u^s,\nonumber
\end{equation}
\begin{equation}
(u^{-r}z)^s=u^{-sr}z^s=(u^{-1}z)^s,\nonumber
\end{equation}
\begin{equation} (u^r)^{sk}=u^{ks}=1\quad\mbox{and}\nonumber
\end{equation}\begin{equation} (u^{-r}z)^{sk}=u^{ks}z^s=(u^{-1}z)^s,\nonumber
\end{equation} since $sr\equiv s\pmod m$. Finally, 
\begin{equation}
(u^r)^{E_m(r,n)}=u^{r(1+r+\ldots+r^{n-1})}=u^{r+r^2+\ldots+r^{n-1}+1}=u^{E_m(r,n)},\nonumber
\end{equation}\begin{equation}
(u^{-r}z)^{E_m(r,n)}=u^{-r(1+r+\ldots+r^{n-1})}z^{E_m(r,n)}=u^{-(r+r^2+\ldots+r^{n-1}+1)}z^{E_m(r,n)}=(u^{-1}z)^{E_m(r,n)},\nonumber
\end{equation}\begin{equation}
(u^r)^{E_m(r,o(b))}=u^{E_m(r,o(b))}=1\quad\mbox{and}\nonumber\end{equation}\begin{equation}
(u^{-r}z)^{E_m(r,o(b))}=(u^{-1}z)^{E_m(r,o(b))}=1,\nonumber\end{equation} since $n\mid o(b)$ and $r^n\equiv 1\pmod m$.


\vspace{.2cm}
{\bf Step 2}: $a^f$ and $b^f$ are automorphisms of $A$.

We prove that $a^f$ and $b^f$ are epimorphisms since $A$ is a finite group and, $\langle (u)a^f, v,w,z \rangle\leq A$ and $\langle (u)b^f, v,w,z \rangle\leq A$, where $a^f$ and $b^f$ are homomorphisms of $A$ that fix $v, w$ and $z$.

\indent Now, $\langle (u)a^f, v \rangle=\langle u, v\rangle$, by the definition of the map $a^f$ and, $\langle (u)b^f, v\rangle=\langle u^r, v \rangle=\langle u, v\rangle$, since $(m,r)=1$.

Therefore, $a^f$ and $b^f$ are automorphisms of $A$.

\vspace{.2cm}
{\bf Step 3:} $a^f$ and $b^f$ satisfy the relations of $G$.


The application $a^f$ also fixes $u$, which implies that $(u)[a^f]^\alpha=u$, for any $\alpha\in \mathbb{Z}_+$. In particular, for $\alpha=m$, the identity $(u)[a^f]^m=u$ implies that $$[a^f]^m=Id_A.$$
\indent We will verify that $a^f$ and $b^f$ also satisfy the other relations of $G$. Observe that:
\begin{eqnarray}
(u)[b^f]^2 &=& (u^r)b^f=u^{r^2},\nonumber\\
{(u)[b^f]}^3 &=& (u^{r^2})b^f=u^{r^3}.\nonumber
\end{eqnarray}
\indent Inductively, we have \begin{equation}
(u)[b^f]^{\beta}=u^{r^\beta}, \mbox{\, for any }  \beta\in \mathbb{Z}_+.\label{eq.homo.}\nonumber
\end{equation} In particular, if $\beta=n$ and $\alpha=s$, we have
$$(u)[a^f]^s=u=u^{r^n}=(u)[b^f]^n,$$ since $r^n\equiv 1\pmod m$. This means that $[a^f]^s=[b^f]^n$.

Finally, since $b^{-1}=b^{o(b)-1}$ then $(u)[b^f]^{-1}=u^{r^{o(b)-1}}$. Hence, \begin{eqnarray} (u)[(b^f)^{-1}\circ a^f\circ b^f]&=&\{[(u)b^f] a^f\} (b^f)^{-1}\nonumber\\ &=&((u^r)[a^f])[b^f]^{-1}=(u^{r})[b^f]^{-1}\nonumber\\ &=&u^ {r^{o(b)-1}r}=u^{r^{o(b)}}.\nonumber \end{eqnarray} But $n\mid o(b)$ and $r^n \equiv 1\pmod m$, i.e., $(u)[(b^f)^{-1}\circ a^f\circ b^f]=u^r=(u)[a^f]^r$.

Therefore, $[b^f]^{-1}\circ a^f\circ b^f=[a^f]^r$.

Thus, we conclude that $a^f$ and $b^f$ satisfy the relations of $G$, i.e., \begin{equation}
[a^f]^m=Id_A, \qquad [b^f]^ n=[a^f]^s, \qquad [b^f]^{-1}\circ [a^f]\circ [b^f]=[a^f]^r.\nonumber
\end{equation}


In other words, $f$ determines an action of $G$ on $A$, proving the {\bf Assertion 1}.

According to this action we have an extension $H=A\rtimes_f G$, whose presentation is given by:
\begin{eqnarray} H = \langle a, b, u, v, w, z \hspace{-.2cm}&|& \hspace{-.2cm} a^m=1, b^n=a^s, [ a,b]=a^{r-1}, u^a=u, u^b=u^r,\nonumber\\ &&\hspace{-.2cm}v^{o'(a)}=1, w^{n(k,r-1)}=1,z^{(o'(a),o'(b), sk,E_m(r,o'(b)))}=1, \nonumber \\ &&\hspace{-.2cm } u^{s}=w^n=(u^{-1}z)^s, u^{E_m(r,n)}=v^s=(u^{-1}z)^{E_m(r,n)}, \nonumber \\ && \hspace{-.2cm} [u,v]=[u,w]=[u,z]=[v,w]=[v,z]=[w,z]=1 \nonumber \\
&& \hspace{-.2cm} u^{(m,E_{m}(r,o(b)), sk)}= 1, (v, w, z\ central)
\rangle, \end{eqnarray}

The action $f$ of $G$ on $A$ also induces an application $\epsilon: G^\varphi\rightarrow \Aut(H)$ given by:

\vspace{.5cm}
\begin{minipage}{.1cm}
	\vspace{-.67cm}	
	\begin{eqnarray}\hspace{-5cm}	
	a^\varphi&\longmapsto&a^{\varphi\epsilon}:H\longrightarrow H \nonumber\\
	&& \hspace{1.05cm} a\longmapsto (a)a^{\varphi\epsilon}:=av\nonumber\\
	&& \hspace{1.05cm} b\longmapsto (b)a^{\varphi\epsilon}:=b(u^{-1}z)\nonumber\\
	& &\hspace{1.05cm} u\longmapsto (u)a^{\varphi\epsilon}:=u\nonumber\\
	& &\hspace{1.05cm} v\longmapsto (v)a^{\varphi\epsilon}:=v\nonumber\\
	& &\hspace{.95cm} w\longmapsto (w)a^{\varphi\epsilon}:=w\nonumber\\
	& &\hspace{1.05cm} z\longmapsto (z)a^{\varphi\epsilon}:=z\nonumber
	\end{eqnarray}	
\end{minipage}\hfill
\begin{minipage}{7.85cm}
	\vspace{-.63cm}
	\begin{eqnarray}
	b^\varphi&\longmapsto&b^{\varphi\epsilon}:H\longrightarrow H \nonumber\\
	&& \hspace{1.05cm} a\longmapsto (a)b^{\varphi\epsilon}:=au\nonumber\\
	&& \hspace{1.05cm} b\longmapsto (b)b^{\varphi\epsilon}:=bw\nonumber\\
	& &\hspace{1.05cm} u\longmapsto (u)b^{\varphi\epsilon}:=u^r\nonumber\\
	& &\hspace{1.05cm} v\longmapsto (v)b^{\varphi\epsilon}:=v\nonumber\\
	& &\hspace{.95cm} w\longmapsto (w)b^{\varphi\epsilon}:=w\nonumber\\
	& &\hspace{1.05cm} z\longmapsto (z)b^{\varphi\epsilon}:=z\nonumber
	\end{eqnarray}
	
\end{minipage}

\vspace{.5cm}
\noindent {\bf Assertion 2}: The application $\epsilon$ determines a homomorphism from $G^\varphi$ to $\Aut(H)$.

\vspace{.2cm}
{\bf Step 1}: $a^{\varphi\epsilon}$ and $b^{\varphi\epsilon}$ extend to homomorphisms and preserves the relations of $H$.

Since the application $a^{\varphi \epsilon}$ fixes $u$, $v$, $w$ and $z$ of $H$, we must evaluate whether the relations of $H$ are satisfied for the elements $av$, $b(u^{-1}z)$. The application $b^{\varphi \epsilon}$ fixes $v$, $w$ and $z$ of $H$ and thus, we only need to verify whether the relations of $H$ are satisfied for the elements $au$, $bw$, $u^r$ and $u^{-r}z$ knowing that $[u,v]=[u,w]=[u,z]=[v,w]=[v,z]=[w,z]=1$ in $H$:

\noindent For $a^{\varphi\epsilon}$: $$(av)^m=v^ma^m=1.$$ \indent As \begin{equation} (bu^{-1}z)^ {ns}=w^{n^2}b^{ns} \,\,\, \mbox{and}\,\,\, (av)^{s^2}=v^{s^2}a^{ s^2}=v^{s^2}b^{ns},\nonumber \end{equation} where \begin{equation} (v^s)^s=(u^{E_m(r,n)} )^s=u^{s(1+r+\ldots+r^{n-1})}=u^{sn}=w^{nn},\nonumber \end{equation} by the relations of $H$ and the fact that $sr \equiv s \pmod m$, it follows that $w^{n^2}=v^{s^2}$ and $(av)^s=(bu ^{-1}z)^n$.

In the equation \begin{equation}
(av)^{bu^{-1}z}=(a^rv)^{u^{-1}}=(a^ {u^{-1}})^r,\nonumber\end{equation} we have $(a^u)=a$ such that $a=(a^u)^{u^{-1}} =a^{u^{-1}}$ with $v^{r}=v\in H$. Therefore, $(av)^{bu^{-1}z}=a^rv=a^rv^r=(av)^r$ and then
\begin{equation}
(u)^{av}=u^a=u\quad \mbox{and}\quad (u)^{b(u^{-1}z)}=u^b=u^r\nonumber.
\end{equation}

\noindent For $b^f$: \begin{equation} (au)^m=u^ma^m=1.\nonumber\end{equation} \begin{equation} (bw)^{n}=w^{n}b^{n}=u^sa^s=(au)^s\nonumber\end{equation} \begin{equation} (au)^{bw}=(a^ru^b)=a^ru^r =(au)^r.\nonumber\end{equation} \indent By the relations involving $u^r$ and $u^{-r}z$ in $H$: \begin{equation}(u^r)^m=u^{rm}=1.\nonumber\end{equation} \begin{equation}((u^r)^{-1}z)^m=u^{-rm}z^m=1\nonumber\end{equation}\indent and $sr\equiv s\pmod m$, we have\begin{equation}(u^r)^s=u^{sr}=u^s,\nonumber\end{equation}\begin{equation}((u^r)^ {-1}z)^s=u^{-sr}z^{sr}=(u^{-1}z)^{s}=w^n,\nonumber\end{equation}\begin{equation} (u^r)^{sk}=u^{srk}=u^{sk}=1\quad\mbox{and}\nonumber \end{equation}\begin{equation}(u^{-r}z)^{sk}=u^{-srk}z^{sk}=(u^{-1}z)^{sk}=1.\nonumber\end{equation} Furthermore, \begin{equation} (u^r)^{E_m(r,o(b))}=u^{rE_m(r,o(b))}=u^{E_m(r,o(b))},\nonumber\end{equation}\begin{equation}((u^r)^{-1}z)^{E_m(r,o(b))}=u^{-r{E_m(r,o( b))}}z^{E_m(r,o(b))}=(u^{-1}z)^{E_m(r,o(b))},\nonumber\end{equation} \begin{equation} (u^r)^{E_m(r,n)}=u^{rE_m(r,n)}=u^{E_m(r,n)}\quad\mbox{and}\nonumber \end{equation}\begin{equation}((u^r)^{-1}z)^{E_m(r,n)}=u^{-r{E_m(r,n)}}z^{E _m(r,n)}=(u^{-1}z)^{E_m(r,n)},\nonumber\end{equation} since $n\mid o(b)$ and $r^n\equiv 1\pmod m$. Finally,
\begin{eqnarray}
&(u^r)^{au}=(u^r)^a=u^r\quad\mbox{and}\quad (u^r)^{bw}=(u^b)^r=(u^r)^r\nonumber,&\\
&(u^{-r}z)^{au}=(u^{-r}z)^a=u^{-r}z\quad\mbox{and}\quad (u^{-r}z)^{bw}=(u^b)^{-r}z=(u^{-r})^rz&\nonumber.
\end{eqnarray}
\indent Therefore, $a^{\varphi\epsilon}$ and $b^{\varphi\epsilon}$ extend to homomorphisms of $H$.

\vspace{.2cm}
{\bf Step 2}: $a^{\varphi\epsilon}$ and $b^{\varphi\epsilon}$ are automorphisms of $H$.

The application $a^{\varphi\epsilon}$ fixes $u$, $v$, $w$ and $z$ of $H$. Since $a=avv^{-1}$ and $b=b(u^{-1}z)(u^{-1}z)$, it follows that $H=\langle (a)a^ {\varphi\epsilon},(b)a^{\varphi\epsilon}, (u)a^{\varphi\epsilon}, v, w, z \rangle.$ 

\indent On the other hand, we have that $( u)b^{\varphi\epsilon}=u^r$ with $(m,r)=1$. Then, there exists an integer such $u=u^{\mu r}$.

Therefore, $(u^r)^{\mu}=u^{\mu r}=u$ and then $\langle (u)b^{\varphi\epsilon}, v, w, z \rangle =\langle u, v, w, z \rangle$. By analogy, $\langle (a)b^{\varphi\epsilon},(b)b^{\varphi\epsilon}, (u) b^{\varphi\epsilon}, v, w, z \rangle=H$, since $a=auu^{-1}$ and $b=bww^{-1}$.

\vspace{.2cm}
{\bf Step 3:} $a^{\varphi\epsilon}$ and $b^{\varphi\epsilon}$ satisfy the relations of $G^\varphi$.

Inductively, it is possible to verify that \begin{eqnarray} (a)[a^{\varphi\epsilon}]^{\alpha}&=& av^{\alpha}\nonumber \\
 (b)[a^{\varphi\epsilon}]^{\beta} & = & b(u^{-1}z)^{\beta}\nonumber\\
 (u)[b^{\varphi\epsilon}]^{\gamma} & = & u^{r^\gamma} \nonumber \\
 (a)[b^{\varphi\epsilon}]^{\delta} & = & au^{E_m(r,\delta)} \nonumber \\
 (b)[b^{\varphi\epsilon}]^{\zeta} & = & bw^{\zeta},\nonumber
 \end{eqnarray} for any $\alpha, \beta, \gamma, \delta, \zeta \in \mathbb{Z}_+$.

In particular, $[a^{\varphi\epsilon}]^m =Id_H$ for $\alpha=\beta=m$ since $$\begin{cases} (a)[a^{\varphi\epsilon}]^m=av^m=a\\ (b)[ a^{\varphi\epsilon}]^m=b(u^{-1}z)^m=b, \end{cases}$$ 

Furthermore, taking $\gamma=\delta=\zeta=n$ and $\alpha=s$, $$\begin{cases} (u)[a^{\varphi\epsilon}]^s=u=u^{r^n}=(u)[b^{\varphi\epsilon}]^n\\ 
(a)[a^{\varphi\epsilon}]^s=av^s=au^{E_m(r,n)}=(a)[b^{\varphi\epsilon}]^n\\ (b)[a^{\varphi\epsilon}]^s=b(u^{-1}z)^s=bw^n =(b)[b^{\varphi\epsilon}]^n, \end{cases}$$, which implies that $[a^{\varphi\epsilon}]^s=[b^{\varphi\epsilon}]^n$.

Finally, as $b^{o(b)-1}=b^{-1}$, the following identities hold: \begin{eqnarray}
(u)[b^{\varphi\epsilon}]^{o(b)-1}\circ[a^{\varphi\epsilon}]\circ[b^{\varphi\epsilon}]&=& \left[((u)b^{\varphi\epsilon})a ^{\varphi\epsilon}\right]{b^{\varphi\epsilon}}^{(o(b)-1)}\nonumber\\ 
&=&\left[ ((u)a^{\varphi \epsilon})^r \right]{b^{\varphi\epsilon}}^{(o(b)-1)} =(u^{r^2})[{b^{\varphi\epsilon}}^{(o( b)-1)}]\nonumber\\ 
&=& ((u)[{b^{\varphi\epsilon}}^{(o(b)-1)}])^{r^2}=u ^{r^{o(b)-1}rr}\nonumber\\
&=&u^{r}=(u)[a^{\varphi\epsilon}]^r\nonumber \end{eqnarray} \begin{eqnarray}
(a)[b^{\varphi\epsilon}]^{o(b)-1}\circ[a^{\varphi\epsilon}]\circ[b^{\varphi\epsilon}] & =& \left(((a)b^{\varphi\epsilon})a^{\varphi\epsilon}\right)[{b^{\varphi\epsilon}}^{(o(b)-1)}] \nonumber\\
&=&\left( ((a)a^{\varphi\epsilon})((u)a^{\varphi\epsilon}) \right)[{b^{\varphi\epsilon}} ^{(o(b)-1)}]\nonumber\\ 
&=& (auv)[{b^{\varphi\epsilon}}^{(o(b)-1)}]\nonumber\\ 
& =& (a)[{b^{\varphi\epsilon}}^{(o(b)-1)}](u)[{b^{\varphi\epsilon}}^{(o(b)-1) }](v)[{b^{\varphi\epsilon}}^{(o(b)-1)}]\nonumber\\ 
&=&au^{(r-1)E_m(r,(o(b )-1))+1+E_m(r,(b)-1)}v\nonumber\\ &=&au^{E_m(r,(o(b)))}v=av^r\nonumber\\
&=&(a)[a^{\varphi\epsilon}]^r\nonumber \end{eqnarray} 
\indent A similar analysis applies for $(b)[b^{\varphi\epsilon}]^{o(b)-1}\circ[a^{\varphi\epsilon}]\circ[b^{\varphi\epsilon}]=(b )[a^{\varphi\epsilon}]^r$.

In this way, we conclude that $a^{\varphi\epsilon}$ and $b^{\varphi\epsilon}$ satisfy the relations of $G^\varphi$, i.e., \begin{equation} [a^{\varphi\epsilon}]^m=Id_H, \qquad [b^{\varphi\epsilon}]^n=[a^{\varphi\epsilon}]^s, \qquad [b^{{\varphi\epsilon}}]^{-1}\circ[a^{\varphi\epsilon}]\circ [b^{\varphi\epsilon}]=[a^{\varphi\epsilon}]^r.\nonumber\end{equation}

Therefore we obtain, via the homomorphism $\epsilon: G^\varphi\rightarrow Aut(H)$, the semidirect product of $H$ by $G^\varphi$ given by $$K=H\rtimes_{\epsilon}G^\varphi=(A\rtimes_fG)\rtimes_{\epsilon} G^\varphi.$$
\indent Analogously to the presentation of the group $H$, we obtain the following  presentation for $K$: 
\begin{eqnarray}
\hspace{-.1cm}\langle a, b, u, v, w, z, a^\varphi, b^\varphi \hspace{-.2cm}&\mid& \hspace{-.2cm} a^m=1, b^n=a^s, [a,b]= a^{r-1}, (a^\varphi)^m=1, (b^\varphi)^n=(a^\varphi)^s, v^{o'(a)}=1,\nonumber\\
&&\hspace{-.2cm} [ a^\varphi,b^\varphi]=(a^\varphi)^{r-1},  u^b=u^r,u^{(m,E_{m}(r,o(b)), sk)}=1, w^{n(k,r-1)}=1,\nonumber\\
&& \hspace{-.2cm} z^{(o'(a),o'(b),sk,E_m(r,o'(b)))}=1,u^ {s}=w^n=(u^{-1}z)^s,u^a=u,a^{a^\varphi}=au,\nonumber \\ &&\hspace{-.2cm} b^{a^\varphi }=b(u^{-1}z), u^{E_m(r,n)}=v^s=(u^{-1}z)^{E_m(r,n)},u^{ a^\varphi}=u,v^{a^\varphi}=v,\nonumber \\
&&\hspace{-.2cm} w^{a^\varphi}=w, z^{a^\varphi}=z, a^{b^\varphi}=au, b^{b^\varphi }=bw, u^{b^\varphi}=u^r, v^{b^\varphi}=v, w^{b^\varphi}=w,\nonumber \\ &&\hspace{-.2cm }z^{b^\varphi}=z,[u,v]=[u,w]=[u,z]=[v,w]=[v,z]=[w,z]=1, \nonumber \\ &&\hspace{-.2cm} (v, w, z\ {central})\nonumber \rangle.
\end{eqnarray}
Note that the actions of $G$ on the (normal) subgroup $\langle u, v, w, z \rangle$ is the same as that of $G^\varphi$. Consequently, replacing $a$ with $x_1$, $b$ with $y_1$, $a^\varphi$ with $x_2$ and $b^\varphi$ with $y_2$, we conclude that $K$ has precisely the same presentation as the group $M$ in Theorem~\ref{thm: 1}. This concludes the proof.
\end{proof}

The technique used to prove  {Theorem \ref{TEO.A}} has been used elsewhere in the literature (see for instance [\cite{Rocco2},\cite{Rocco1}]). From Corollary~\ref{cor: nats}  and the fact that the non abelian exterior square $G \wedge G \simeq \Upsilon(G)/\Delta(G)$, while $\mu(G)/\Delta(G) \simeq M(G)$ (cf. \ref{eq:mu}), we have  
\begin{cor}\label{coro.ordem.quadd.meta.fini} Let $G=g(a,b;m,n,r,s)$ be a finite metacyclic group with $m$ odd. Then,  
 \begin{itemize}
 \item[$(i)$] $G\wedge G =
	\langle \bar{u} \mid \bar{u}^{(m,E_m(r,n),s)}=1\rangle,$ where $\bar{u}:=u\Delta(G)$,
 \item[$(ii)$] $M(G) =
	\langle \bar{u} \mid  \bar{u}^{\frac{(m,r-1)}{m}(m,E_{m}(r,n), s)}=1\rangle,$ where $\bar{u}:=u^{\frac{m}{(m,r-1)}}\Delta(G)$,
\end{itemize}
	where $u$ stands for the generator of $G \otimes G$ corresponding to the generator $[a, b^\vfi]$ of $\Upsilon(G)$.  
\end{cor}

The following Corollary gives  presentations for the groups $\nu(G)$, $[G,G^\varphi]$, $G \wedge G$ and $M(G)$, in the particular case where $G=g(a,b;m,n,m-1,0)$ is a finite split metacyclic group with $m$ odd.

\begin{cor}\label{coro.teo.a.split} Let $G=g(a,b; m,n, m-1,0)$ be a finite split metacyclic group with $m$ odd. Then
\begin{itemize}
\item[(i)] A presentation for the group $\nu(G)$ is given by:
\begin{eqnarray}\langle a,b,a^\vfi,b^\vfi, u, v, w, z \hspace{-.2cm}&|& \hspace{-.2cm} a^m=1, (a^\vfi)^m=1, b^n=1, (b^\vfi)^n=1, [a^\vfi,b^\vfi]=(a^\vfi)^{r-1}, w^{n}=1,\nonumber \\ &&\hspace{-.2cm} [a,b^\vfi]=u, 
	[a,a^\vfi]=v, [b,b^\vfi]=w, [b, a^\vfi]=u^{-1}z, v^{o'(a)}=1,  \nonumber \\ &&\hspace{-.2cm}  z^{(m,r-1,n,E_m(r,n))}=1, u^{a}=u^{a^\vfi}=uv^{r-1}, [a,b]=a^{r-1}, \nonumber \\ &&\hspace{-.2cm} u^{b}=u^{b^\vfi}=u^{r}v^{(r-1)\binom{r}{2}},u^{(m,E_m(r,n))}=1,\ (v, w, z\ central) \nonumber\rangle,
	\end{eqnarray} 
    \item[(ii)] \begin{eqnarray}\hspace{-.2cm}\,\ G \otimes G \simeq \langle u, v, w , z \hspace{-.2cm}& | & \hspace{-.2cm} w^{n}=1, u^{m}=1, [u,v ]=[v,w]=[v,z]=[w,z]=1,\nonumber \\ &&\hspace{-.2cm}v=1,z=1,[u,w]=[u ,z]=1\nonumber\rangle, 
    \end{eqnarray}
 \item[(iii)]\,\ $G\wedge G \simeq 
	\langle \bar{u} \mid \bar{u}^{m}=1\rangle,$ where $\bar{u}:=u\Delta(G)$, 
	\item[(iv)]\,\ $M(G) \simeq
	\langle \bar{u} \mid  \bar{u}^{(m,m-2)}=1 \rangle,$ where $\bar{u}:=u^{\frac{m}{(m,m-2)}} \Delta(G)$. 
    \end{itemize}
 \end{cor}

In $1970$, Wamsley [\cite{Wamsley}, Lemma $1$], showed that finite, unsplit metacyclic groups have  cyclic Schur multiplier whose order depends on the 4-tuple $(m,n,r,s)$. The second named author  \cite{Rocco2}, in $1994$, also computed $M(G)$ for a general finite metacyclic group $G$. His calculations are based on the isomorphism $M(G)\cong\mu(G)/\Delta(G)$, taking into account that $G'$ is a cyclic group isomorphic to the quotient $[G, G^\varphi]/\mu(G)$. In other words, the relative orders of the generating elements of the Schur multiplier, certified by isomorphism, are calculated modulo $\Delta(G)$. Our results agree with the order of $M(G)$ calculated by both, regardless of the parity of $m$.

Beyl [\cite{Beyl}, page $147$, $\S5$] classified the so-called metacyclic Schur $p$-groups, metacyclic $p$-groups that have a trivial Schur multiplier, up to isomorphism. He also showed that groups of the type \begin{equation} \langle a,b \mid a^m=1, b^n=a^{\frac{m}{(m,r-1)}}, a^b=a^r \rangle\nonumber\end{equation}
have a trivial Schur multiplier. This result can be easily obtained from our corollaries by taking $s=\frac{m}{(m,r-1)}$.

The abelian sections of the group $\nu(G)$ for split metacyclic groups described by Johnson \textnormal{[\cite{BJR}, $(16)$ Proposition]} when $m$ is even, as well as those described by Brown, Johnson \& Robertson \textnormal{[\cite{BJR}, Proposition $15$]}, when $m$ is odd, both obtained from a biderivation approach, coincide with those we have obtained in this article.

Independently of the parity of $m$, Brown, Johnson \& Robertson [\cite{BJR}, Proposition $12$] computed the orders of a generating set for the group $G \otimes G$ when $G=g(a,b;m,2,m-1,0)$. The strategy used by Brown, Johnson \& Robertson [\cite{BJR}] was to consider the dihedral group $\mathcal{D}_m$ as the central extension of the group of generalized quaternions $\mathcal{Q}_{m}$. Our bound also satisfies the proposal in \textnormal{[\cite{Johnson}], Proposition $16$]}, independent of the parity of $n$ and $r$.\\

\noindent {\bf Declaration:} 
The authors declare that there is no conflict of interest. 
 

\renewcommand{\refname}{REFERENCES}


\begin{thebibliography}{99}


\bibitem{Bacon} M.\,R. Bacon, {\it On the non-abelian tensor square of a nilpotent of a group}, Glasgow Math. J., {\bf 36} (1994), pp. 291--296. 

\bibitem{BK} M.\,R. Bacon and L.-C. Kappe, {\it The nonabelian tensor square of a $2$-generator $p$-group of class $2$}, Arch. Math., {\bf 61} (1993), pp. 508--516.

\bibitem{BKM} M. Bacon, L.-C. Kappe and R. F. Morse, {\it On the nonabelian tensor square of a $2$-Engel group}, Arch. Math., {\bf 69} (1997), pp. 353--364.

\bibitem{Kappe1} J.R. Beuerle; L.C. Kappe, {\it Infinite metacyclic groups and their non-abelian tensor squares.}\ {Proceedings of the Edinburgh Mathematical Society} $43$, no. $3$ ($2000$): $651-662$.


\bibitem{Beyl} F.R. Beyl, {\it The Schur multiplicator of metacyclic groups.} {Proceedings of the American Mathematical Society} $40$, no. $2$ ($1973$): $413-418$. 



\bibitem{BFM} R.\,D. Blyth and F.\ Fumagalli, M.\ Morigi, {\it Some structural results on the non-abelian tensor square of groups}, J. Group Theory, {\bf 13} (2010) 83--94.

\bibitem{BlyMor} R.D. Blyth, R.F. Morse, {\it Computing the nonabelian tensor squares of polycyclic groups} J. Algebra, {\bf 321} (2009), pp. 2139--2148.



\bibitem{BJR} R. Brown, D.\,L. Johnson, E.\,F. Robertson, {\em Some computations of non-abelian tensor products of groups}, J. Algebra {\bf 111} (1987) 177--202.

\bibitem{BL1}  R. Brown, J.-L. Loday,  {\it Excision homotopique en base dimension}, C.R. Acad. Sci. Paris S.I Math. {\bf 298}, No. 15 (1984) 353--356.

\bibitem{BL} R. Brown, J.-L. Loday, {\it Van Kampen Theorems for Diagrams of Spaces}, Topology, {\bf 26} (1987) 311--335.



\bibitem{EN}  B. Eick,  W. Nickel, {\it Computing the Schur multiplicator and the nonabelian tensor square of a polycyclic group}, J. Algebra {\bf 320}, No.2 (2008) 927--944.

 \bibitem{EL} G. Ellis, F. Leonard, {\it Computing Schur multipliers and tensor products of finite groups}, Proc. Royal Irish Acad., {\bf 95A} (1995) 137--147.
	

\bibitem{Johnson} D.L. Johnson, "The nonabelian tensor square of a finite split metacyclic group."\ \textit{Proceedings of the Edinburgh Mathematical Society} $30$, no. $1$ ($1987$): $91-95$.


\bibitem{Kappe} L.-C. Kappe, {\em Nonabelian tensor products of groups:  the commutator connection}, Proc. Groups St. Andrews 1997 at Bath, London Math. Soc. Lecture Notes, {\bf 261} (1999), pp. 447--454.




\bibitem{MM} A. Magidin and R.\,F. Morse, {\it Certain homological functors for $2$-generator $p$-groups of class two}, Computational Group Theory and the Theory of Groups, II, Contemporary Mathematics, {\bf 511} (2010), pp. 127--166.





\bibitem{Miller} C. Miller, "The second homology group of a group; relations among commutators."\ \textit{Proceedings of the American Mathematical Society }$3$, no. $4$ ($1952$): $588-595$.



\bibitem{Rob} D.J. Robinson, \textit{A Course in the Theory of Groups}. Vol. $80$. Springer Science \& Business Media, $2012$.

\bibitem{Rocco2} N.R. Rocco, "A presentation for a crossed embedding of finite solvable groups."\ \textit{Communications in Algebra }$22$, no. $6$ ($1994$): $1975-1998$.



\bibitem{Rocco1} N.R. Rocco, "On a construction related to the non-abelian tensor square of a group." \textit{Boletim da Sociedade Brasileira de Matem\'atica-Bulletin/Brazilian Mathematical Society} $22$, no. $1$ ($1991$): $63-79$.



\bibitem{Sim} H.S. Sim, "Metacyclic groups of odd order."\ \textit{Proceedings of the London Mathematical Society }$3$, no. $1$ ($1994$): $47-71$.



\bibitem{V} M.\,P. Visscher, {\it On the nilpotency class and solvability length of the nonabelian tensor product of groups}, Arch. Math., {\bf 73} (1999), pp. 161--171.

\bibitem{Wamsley} J.W. Wamsley, "The deficiency of metacyclic groups."\ \textit{Proceedings of the American Mathematical Society }$24$, no. $4$ ($1970$): $724-726$.


\bibitem{Za} H. Zassenhaus, \textit{The theory of groups.} Courier Corporation, $1999$.

\end{thebibliography}
\end{document}